\documentclass{article}
\usepackage[utf8]{inputenc}
\usepackage{amsmath, amsthm}
\usepackage{amssymb}
\usepackage{xcolor}
\usepackage{pgf,tikz,pgfplots}  
\usepackage{graphicx}
\usepackage{hyperref}
\usetikzlibrary{arrows}

\newcommand{\norm}[1]{\left\lVert#1\right\rVert}
\newcommand{\comentario}[1]{\color{blue}#1\color{black}}

\newtheorem{theorem}{Theorem}
\newtheorem{prop}[theorem]{Proposition}
\newtheorem{lemma}[theorem]{Lemma}
\newtheorem{coro}[theorem]{Corollary}
\newtheorem{prob}{Problem}

\title{On prescribing total preorders and linear orders to pairwise distances of points in Euclidean space}
\author{Víctor Hugo Almendra-Hernández\\Leonardo Martínez-Sandoval}
\date{}

\begin{document}

\maketitle

\begin{abstract}
    We show that any total preorder on a set with $\binom{n}{2}$ elements coincides with the order on pairwise distances of some point collection of size $n$ in $\mathbb{R}^{n-1}$. For linear orders, a collection of $n$ points in $\mathbb{R}^{n-2}$ suffices. These bounds turn out to be optimal. We also find an optimal bound in a bipartite version for total preorders and a near-optimal bound for a bipartite version for linear orders. Our arguments include tools from convexity and positive semidefinite quadratic forms.
\end{abstract}


\section{Introduction}

The study of distances among points in Euclidean space is a topic in discrete geometry that has stimulated a lot of research and has interconnected several areas of mathematics. Two striking examples are Erd\H{o}s problems on unit distances  and distinct distances \cite{erdos1946sets}.




In this note, we study the following combinatorial problem on Eucliean distances. For a positive integer $n$, we define $[n]=\{1,2\ldots,n\}$. Let $P$ be a collection of points $p_i$ for $i\in [n]$ in $d$-dimensional Euclidean space in general position.

Under these assumptions, the point collection $P$ induces a total preorder on the family of pairs $$D_n= \binom{[n]}{2}=\{(i,j):1\leq i < j \leq n\},$$ given by $(i_1,i_2)\leq (j_1,j_2)$ if and only if $\norm{p_{i_1}-p_{i_2}}\leq \norm{p_{j_1}-p_{j_2}}$. When $P$ induces pairwise distinct distances, this preorder is also antisymmetric and thus it is a linear order on $D_n$.

Is every total preorder on $D_n$ obtainable in this way? What about linear orders?

We will show that every given total preorder or linear order is achievable if and only if $d$ is large enough with respect to $n$. Our main result is an optimal bound on the minimal dimension required for this to happen.

\begin{theorem}
\label{thm:main}
  Let $n\geq 3$ be an integer.
  \begin{itemize}
      \item The minimal dimension into which any linear order on $D_n$ can be induced by the pairwise distances of a point collection in $\mathbb{R}^d$ is $d=n-2$.
      \item The minimal dimension into which any total preorder on $D_n$ can be induced by the pairwise distances of a point collection in $\mathbb{R}^d$ is $d=n-1$.
  \end{itemize}
\end{theorem}

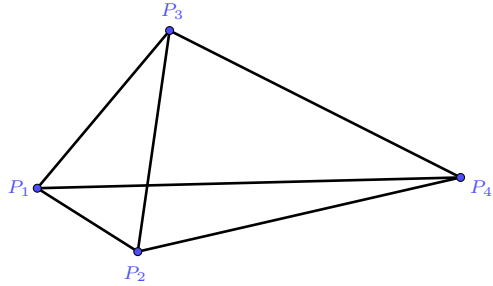
\definecolor{ududff}{rgb}{0.30196078431372547,0.30196078431372547,1}
\begin{figure}[h!]
\begin{center}
\begin{tikzpicture}[line cap=round,line join=round,>=triangle 45,x=0.3cm,y=0.3cm]
\clip(-8,-7.5) rectangle (28,8.5);
\draw [line width=1pt] (5.86374,6.99236)-- (0,0);
\draw [line width=1pt] (0,0)-- (4.45337,-2.8131);
\draw [line width=1pt] (4.45337,-2.8131)-- (18.7586,0.47777);
\draw [line width=1pt] (18.7586,0.47777)-- (5.86374,6.99236);
\draw [line width=1pt] (5.86374,6.99236)-- (4.45337,-2.8131);
\draw [line width=1pt] (0,0)-- (18.7586,0.47777);
\begin{scriptsize}
\draw [fill=ududff] (0,0) circle (1.5pt);
\draw[color=ududff] (-0.8,-0.0072408803097450125) node {$P_{1}$};
\draw [fill=ududff] (4.45337,-2.8131) circle (1.5pt);
\draw[color=ududff] (4.346608530553978,-3.8) node {$P_{2}$};
\draw [fill=ududff] (5.86374,6.99236) circle (1.5pt);
\draw[color=ududff] (6.004094634323034,7.849243251555577) node {$P_{3}$};
\draw [fill=ududff] (18.7586,0.47777) circle (1.5pt);
\draw[color=ududff] (19.7, 0.0072408803097450125) node {$P_{4}$};
\node at (10, -6) {$(1,2) \leq (2,3) \equiv (1,3) \leq (3,4) \equiv (2,4) \leq (1,4)$};
\end{scriptsize}
\end{tikzpicture}
\caption{Example of induced preorder on $D_4$}
\end{center}
\end{figure}

We also study a bipartite version of this problem. Consider collections $P$ and $Q$ of $n$ and $m$ points on $d$-dimensional Euclidean space, respectively. We allow each of $P$ and $Q$ to have repeated points, however we require each point of $P$ to be distinct from each point of $Q$.

If the points of $P$ are $p_i$ for $i\in[n]$ and the points of $Q$ are $q_i$ for $i\in[m]$, now the pair $(P, Q)$ induces a total preorder in the family of pairs $$B_{n,m} = \{(i,j) : i \in [n], j \in [m] \},$$ given by $(i_1,i_2)\leq (j_1,j_2)$ if and only if $\norm{p_{i_1}-q_{i_2}}\leq \norm{p_{j_1}-q_{j_2}}$. 

\begin{figure}[h!]
\begin{tikzpicture}[line cap=round,line join=round,>=triangle 45,x=1cm,y=1cm]
\clip(-10.47,-2.5) rectangle (14.77,2.5);
\draw [line width=1pt] (-6.93,-0.22)-- (-5.15,1.92);
\draw [line width=1pt] (-5.15,1.92)-- (-3.97,1);
\draw [line width=1pt] (-3.97,1)-- (-4.33,-1.36);
\draw [line width=1pt] (-4.33,-1.36)-- (-2.01,2.26);
\draw [line width=1pt] (-2.01,2.26)-- (-5.15,1.92);
\draw [line width=1pt] (-6.93,-0.22)-- (-4.33,-1.36);
\begin{scriptsize}
\draw [fill=ududff] (-4.33,-1.36) circle (1.5pt);
\draw[color=ududff] (-4.1,-1.45) node {$q_2$};
\draw [fill=ududff] (-5.15,1.92) circle (1.5pt);
\draw[color=ududff] (-5.07,2.17) node {$q_1$};
\draw [fill=ududff] (-2.01,2.26) circle (1.5pt);
\draw[color=ududff] (-1.69,2.35) node {$p_1$};
\draw [fill=ududff] (-3.97,1) circle (1.5pt);
\draw[color=ududff] (-3.75,1.21) node {$p_2$};
\draw [fill=ududff] (-6.93,-0.22) circle (1.5pt);
\draw[color=ududff] (-7.1,0) node {$p_3$};
\node at (-4.5, -2.2) {$(2,1) \leq (2,2) \equiv (1,1) \equiv (3,1) \equiv (3,2) \leq (1,2)$};
\end{scriptsize}
\end{tikzpicture}
\caption{Example of induced preorder on $B_{2,3}$}
\end{figure}

We show also an optimal bound for the minimal dimension where every total preorder on $B_{n,m}$ is obtainable in this way. For the linear order case, we give a near-optimal result.

\begin{theorem}
\label{thm:bipartite}
    Let $n \geq 2$ be an integer.
    \begin{itemize}
        \item The minimal dimension $d$ into which any linear order on $B_{n,m}$ can be induced by the pairwise distances between two point collections in $\mathbb{R}^d$ satisfies $\min(n,m)-1 \leq d \leq \min(n,m)$.
        \item The minimal dimension into which any total preorder on $B_{n,m}$ can be induced by the pairwise distances between two point collections in $\mathbb{R}^d$ is $d = \min(n,m)$.
    \end{itemize}
\end{theorem}


As a reminder, a total preorder on a set $X$ is a reflexive and transitive relation $\leq$ in which every two elements of $X$ can be compared. We say that $x<y$ if $x\leq y$ is in the relation but $y\leq x$ is not. We say that $x\equiv y$ if and only if $x\leq y$ and $y\leq x$. It is immediate that $\equiv$ is an equivalence relation on $X$ and that $\leq$ induces a linear order on the equivalence classes. If $\leq$ is antisymmetric, then each equivalence class has exactly one element, so $\leq$ is itself a linear order on $X$.

Note that when $d=2$, the unit distance problem and the distinct distances problem impose restrictions on the size and number of equivalence classes induced by $\leq$, respectively.

We divide the proof of Theorem \ref{thm:main} in two sections. In Section \ref{sec:lower} we provide our lower bounds. We exhibit for each $n\geq 3$ a linear order on $D_n$ that cannot be induced from a family of $n$ points in $\mathbb{R}^{n-3}$. We also exhibit a preorder that cannot be induced from a family of $n$ points in $\mathbb{R}^{n-2}$.

In Section \ref{sec:upper} we use a powerful lemma on Euclidean distances and positive semidefinite matrices to induce any given linear order on $D_n$ from a point family in $\mathbb{R}^{n-2}$. To introduce our technique, first we prove that any preorder on $D_n$ can be attainable in $\mathbb{R}^{n-1}$. We then show how to adapt the technique to reduce the required dimension in the case of linear orders.


In a similar way, we divide the proof of Theorem \ref{thm:bipartite} in two sections. In Section \ref{sec:BipartiteLower} we show our lower bounds. We provide a linear order and a linear preorder on $B_{n,n}$ that cannot be induced from two collections of points with cardinality $n$ in $\mathbb{R}^{n-1}, \mathbb{R}^n$, respectively. This will suffice to obtain the lower bounds since $B_{n,n}$ is a subset of $B_{r,s}$, when $n = \min(r,s)$.

In Section \ref{sec:BipartiteUpper} we adapt the technique showed in Section \ref{sec:upper}.   


\section{Complete case, lower bounds}
\label{sec:lower}

If $n=3$, it is clear that we need $d\geq 1$ to attain every possible linear order on $D_3$. Thus, the first non-trivial case is $n=4$. We claim that there is no point collection in $\mathbb{R}$ that induces a linear order on $D_4$ with the following relations:

\begin{enumerate}
    \item $(1,4)$ is the (unique) maximal element of the linear order
    \item $(1,2)<(1,3)$
    \item $(2,4)<(3,4)$
\end{enumerate}

The proof is quite simple in this case: since $(1,4)$ is the maximal element, this forces $p_1$ and $p_4$ to be the extremal points in the geometric configuration. We may assume without loss of generality that $p_1$ is the leftmost point and $p_4$ is the rightmost point. The second relation forces $p_2$ to be closer to $p_1$ than $p_3$, but this contradicts the last relation.

\begin{figure}[h!]
\begin{center}
\begin{tikzpicture}[line cap=round,line join=round,>=triangle 45,x=1.5cm,y=1.5cm]
    \clip(-1,-0.28) rectangle (4,0.5);
    \draw [line width=1pt] (0,0)-- (3,0);
    \begin{tiny}
    \draw [fill=ududff] (0,0) circle (1.5pt);
    \draw[color=ududff] (-0.0051911975692012224,-0.18) node {$P_{1}$};
    \draw [fill=ududff] (1,0) circle (1.5pt);
    \draw[color=ududff] (0.9778721835887161,-0.18) node {$P_{2}$};
    \draw [fill=ududff] (1.5,0) circle (1.5pt);
    \draw[color=ududff] (1.473997628285235,-0.18) node {$P_{3}$};
    \draw [fill=ududff] (3,0) circle (1.5pt);
    \draw[color=ududff] (2.9623739623747922,-0.18) node {$P_{4}$};
    \end{tiny}
\end{tikzpicture}
\end{center}
\caption{Case n = 4}
\end{figure}

The constructions for larger values of $n$ require a careful selection of prescribed relations and convexity arguments. We begin by proving a geometric auxiliary result and a corollary.

\begin{prop}
  Let $P=\{p_1,\ldots,p_{d+1}\}$ be a set of $d+1$ affinely independent points in $\mathbb{R}^d$. For $i\in [d+1]$ let $\Pi_i$ be the hyperplane spanned by $P\setminus\{p_i\}$ and $H_i$ the closed halfspace defined by $\Pi_i$ in which $p_i$ lies.
  
  Let $\Pi$ be a hyperplane such that $p_1$ lies in one of its open halfspaces, which we call  $H^+$, and such that $P\setminus\{p_1\}$ is contained in the (closed) complement $H^-$. Then the closure $\Delta'$ of $$H^+\cap H_2 \cap \ldots \cap H_{d+1}$$ is a simplex contained in the simplex $$\Delta=H_1\cap H_2 \cap \ldots \cap H_{d+1}.$$
\end{prop}

\begin{proof}
  Let $x$ be a point in $$H^+\cap H_2 \cap \ldots \cap H_{d+1}.$$ 
  
  This point lies in the convex cone $H_2 \cap \ldots \cap H_{d+1}$ that can be partitioned by $\Pi_1$ into $\Delta$ and a region $R$. If $x$ does not lie in $H_1$, then $x$ is in the region $R$ and therefore the segment $p_1x$ intersects the convex hull of $P\setminus\{p_1\}$ in a point $y$.
  
  On the one hand, $y$ would be on the segment $p_1x$, whose endpoints are in $H^+$. This implies that $y$ is in $H^+$. On the other hand, $y$ would be a convex combination of the points in $P\setminus\{p_1\}$, all of which lie in $H^-$. By convexity, $y$ itself would be in $H^-$. This is a contradiction.
  
  We conclude that $x$ lies in $H_1$, and since it originally lies in $H_2 \cap \ldots \cap H_{d+1}$, then it is in $\Delta$. Since $\Delta$ is closed, we conclude $\Delta'\subseteq \Delta$.
  
\end{proof}

\textbf{Edit made on February 2026:} Corollary 4 below is false as stated, noted in blue. For a corrected version, see Lemma 5 from \cite{maldonado25}. As a consequence, the counterexample claimed in Proposition 5 turns out to be realizable in $\mathbb{R}^{n-3}$. See Proposition 6 from \cite{maldonado25} for a set of points in $\mathbb{R}^{n-3}$ with the proposed linear order. This leaves open the question for $\mathbb{R}^{n-3}$ whether any order for pairwise distances on $n$ points is realizable or not.

\textcolor{blue}{By repeatedly applying the lemma above, we have the following consequence.}

\begin{coro}
\label{coro:hyper}
Let $P=\{p_1,\ldots,p_{d+1}\}$ be a set of $d+1$ affinely independent points in $\mathbb{R}^d$. For $i\in [d+1]$ let $\Pi_i$ be a hyperplane such that $p_i$ lies on one of its open halfspaces $H_i$, and $P\setminus\{p_i\}$ is contained in the complement of $H_i$.

Then the closure $\Delta'$ of $H_1\cap H_2 \cap \ldots \cap H_{d+1}$ is a (possibly empty) simplex contained in the simplex spanned by $P$.
\end{coro}

\begin{prop}
  Let $n\geq 4$ be an integer. Then there is no collection of $n$ points $P$ in general position in $\mathbb{R}^{n-3}$ that induces a linear order on $D_n$ including the following in the relation:
  \begin{enumerate}
      \item For any pair $(i,j)$ in $D_{n-3}$ and a pair $(k,l)$ in $D_n\setminus D_{n-3}$, we have $(i,j)>(k,l)$.
      \item For any pair $(i,j) \in [n-3] \times \{n-2,n-1,n\}$ and a pair $(k,l)$ in $D_n \setminus ([n-3] \times \{n-2,n-1,n\})$ we have $(i,j)<(k,l)$.
  \end{enumerate}
\end{prop}

\begin{figure}[h!]
\begin{center}
\definecolor{ffqqqq}{rgb}{1,0.5,0}
\definecolor{qqzzff}{rgb}{0,0.6,1}
\definecolor{qqwuqq}{rgb}{0,0.39215686274509803,0}
\definecolor{ududff}{rgb}{0.30196078431372547,0.30196078431372547,1}
\begin{tikzpicture}[line cap=round,line join=round,>=triangle 45,x=1.2cm,y=1.2cm]
    \clip(-1.2,-2.8) rectangle (6,1.2);
    \draw [rotate around={-85.80233164283867:(0.72,-0.57)},line width=0.5pt,dotted] (0.72,-0.57) ellipse (1.672552329971275cm and 1.2660692305290182cm);
    \draw [line width=1pt,color=qqwuqq] (-0.1,-0.32)-- (1.36,-1.38);
    \draw [line width=1pt,color=qqzzff] (1.24,-0.08)-- (2.84,-0.58);
    \draw [rotate around={-87.24089234137965:(3.26,-0.73)},line width=0.5pt,dotted] (3.26,-0.73) ellipse (1.2098376837290625cm and 0.8793220234765564cm);
    \draw [line width=1pt,color=ffqqqq] (3.5,0.02)-- (3.36,-1.54);
    \begin{scriptsize}
    \draw [fill=ududff] (0.64,0.52) circle (1.5pt);
    \draw [fill=ududff] (0.8,-1.66) circle (1.5pt);
    \draw [fill=ududff] (1.24,-0.08) circle (1.5pt);
    \draw [fill=ududff] (0.26,-1.14) circle (1.5pt);
    \draw [fill=ududff] (-0.1,-0.32) circle (1.5pt);
    \draw [fill=ududff] (1.36,-1.38) circle (1.5pt);
    \draw [fill=ududff] (2.84,-0.58) circle (1.5pt);
    \draw [fill=ududff] (3.5,0.02) circle (1.5pt);
    \draw [fill=ududff] (3.36,-1.54) circle (1.5pt);
    \node at (0.75,-2.5) {$\{P_1, \dots, P_{n-3}\}$};
    \node at (3.4,-2.5) {$\{P_{n-2}, P_{n-1}, P_n\}$};
    \draw [fill=qqzzff] (4.7,-0.8) circle (2.5pt);
    \node at (5, -0.8) { $<$ };
    \draw [fill=ffqqqq] (5.3,-0.8) circle (2.5pt);
    \node at (5.6, -0.8) { $<$ };
    \draw [fill=qqwuqq] (5.9,-0.8) circle (2.5pt);
    \end{scriptsize}
\end{tikzpicture}
\end{center}
\caption{Relations for order in $D_n$}
\end{figure}

\begin{proof}
We proceed by contradiction. Suppose that there is a point collection $P=(p_1,p_2,\ldots,p_n)$ that induces a linear order $\leq$ on $D_n$ with the given relations. Since $\leq$ is a linear order, all points from $P$ are distinct and all pairwise distances of $P$ are also different.

Let $\Pi$ be the affine hyperplane spanned by $p_1,\ldots,p_{n-3}$. Among the points $p_{n-2},p_{n-1},p_n$, two of them lie on the same open halfspace defined by $\Pi$. Without loss of generality, we may assume they are $p_{n-2}$ and $p_{n-1}$. We now show that  $p_{n-1}$ lies in the simplex $\Delta$ spanned by $\{p_1,\ldots,p_{n-2}\}$.

To do so, for $i\in [n-3]$ let $\Pi_i$ be the perpendicular bisector hyperplane to the segment $p_ip_{n-2}$ and $\Pi_{n-2}=\Pi$, the hyperplane spanned by $\{p_1,\ldots,p_{n-3}\}$. For $i$ in $[n-2]$ let $H_i$ be the open halfspace of $\Pi_i$ on which $p_i$ lies.

By the relations in (1), the distances among points $p_i,p_j$ with $i,j\in [n-3]$ are the largest. This implies that for distinct $i,j \in [n-3]$ we have $\norm{p_i-p_j}> \norm{p_j-p_{n-2}}$, so $p_j$ is on the opposite open halfspace defined by $\Pi_i$ as $p_i$. This means that $H_1,\ldots,H_{n-2}$ satisfy the hypothesis of Corollary \ref{coro:hyper}, and therefore the closure $\Delta'$ of $H_1\cap\ldots\cap H_{n-2}$ is contained in $\Delta$.

To finish the proof of our claim, note that the relations in (2) imply that the next largest distances are among points $p_i,p_j$ with $i,j\in\{n-2,n-1,n\}$, so the remaining distances are smaller than these. In particular, for every $i$ in $[n-3]$ we have that $$\norm{p_{n-1}-p_{n-2}}>\norm{p_{n-1}-p_{i}},$$

and thus $p_{n-1}$ lies in $H_i$. Since $p_{n-1}$ was originally chosen to be in $H_{n-2}$, we conclude that $p_{n-1}$ is in $\Delta'\subseteq \Delta$, as claimed.

An analogous proof shows that $p_{n-2}$ lies in the simplex spanned by $P\setminus\{p_{n-2},p_n\}$. We conclude that $p_{n-2}=p_{n-1}$, a contradiction to $P$ having $n$ distinct points.

\end{proof}

\begin{figure}[h!]
\begin{center}
    \definecolor{ffvvqq}{rgb}{1,0.3333333333333333,0}
    \definecolor{ududff}{rgb}{0.30196078431372547,0.30196078431372547,1}
    \begin{tikzpicture}[line cap=round,line join=round,>=triangle 45,x=1cm,y=1cm]
    \clip(-1,-2) rectangle (5.5,3);
    \fill[line width=1pt,color=ffvvqq,fill=ffvvqq,fill opacity=0.3] (1.7008784118237559,0.006181847572183319) -- (1.9737007322255737,0.3868339422280541) -- (2.822433590580517,-0.0005056065383293601) -- cycle;
    \draw [line width=0.8pt,dotted,domain=-2.540024240124981:8.525926953575134] plot(\x,{(-0-0.0007071067530400426*\x)/3.9472627440716916});
    \draw [line width=1pt] (0,0)-- (0.9726538452146883,2.1312651415745);
    \draw [line width=1pt] (0.9726538452146883,2.1312651415745)-- (3.9472627440716916,-0.0007071067530400426);
    \draw [line width=1pt] (0,0)-- (3.9472627440716916,-0.0007071067530400426);
    \draw [line width=0.8pt,dotted,domain=-2.540024240124981:8.525926953575134] plot(\x,{(-2.7441733031507463--0.9726538452146883*\x)/-2.1312651415745});
    \draw [line width=0.8pt,dotted,domain=0.540024240124981:3.525926953575134] plot(\x,{(-5.046268532217425--2.9746088988570034*\x)/2.13197224832754});
    \begin{tiny}
    \draw [fill=ududff] (0,0) circle (1.5pt);
    \draw[color=ududff] (-0.10348452775064385,-0.22608772875693935) node {$P_{1}$};
    \draw [fill=ududff] (3.9472627440716916,-0.0007071067530400426) circle (1.5pt);
    \draw[color=ududff] (4.1096987248966474,-0.22608772875693935) node {$P_{2}$};
    \draw [fill=ududff] (0.9726538452146883,2.1312651415745) circle (1.5pt);
    \draw[color=ududff] (0.9929583428178078,2.390147486014273) node {$P_{3}$};
    \draw [fill=ududff] (3.3381278159781074,1.3698464814575215) circle (1.5pt);
    \draw[color=ududff] (3.4091935575890258,1.6287288258972943) node {$P_{4}$};
    \draw [fill=ududff] (1.67316,-1.05654) circle (1.5pt);
    \draw[color=ududff] (1.6325500173160714,-1.2920738529207095) node {$P_{5}$};
    \end{tiny}
    \begin{scriptsize}
    \draw (4.8,0.2) node {$\Pi$};
    \draw (-0.10348452775064385,1.5) node {$\Pi_1$};
    \draw (3,2.3) node {$\Pi_2$};
    \end{scriptsize}
    \end{tikzpicture}
    \caption{Proof for $n = 5$ in the plane}
\end{center}
\end{figure}

Now we focus on the lower bound for total preorders.

\begin{prop}
\label{prop:lowerpre}
  Let $n\geq 3$ be an integer. Then there is no collection of $n$ points $P$ in $\mathbb{R}^{n-2}$ that induces a preorder on $D_n$ in which $(n-1,n)$ is a unique minimal element and the rest of the pairs belong to a single and maximal class.
\end{prop}

\begin{proof}

Suppose that a point collection $P=(p_1,\ldots,p_n)$ induces the required relations. In particular, we require all distances among the points $P\setminus\{p_{n-1}\}$ to be equal, so these $n-1$ points must be the vertices of a regular simplex. Similarly, $P\setminus\{p_n\}$ must be the vertices of a regular simplex. Without loss of generality, these simplices have side length $1$.

There are only two ways in which these conditions may happen simultaneously. The first one is that $p_{n-1}=p_n$, but this contradicts the fact that $\norm{p_{n-1} - p_{n}} > 0$. The second one is that $p_{n-1}$ is the reflection of $p_n$ with respect to the hyperplane through $P\setminus\{p_{n-1},p_n\}.$ But in this case, a standard calculation for the height of the simplex implies that

$$\norm{p_{n-1}-p_n}=2\sqrt{\frac{n-1}{2(n-2)}}>1=\norm{p_1-p_2},$$

which contradicts the required relation $(n-1,n)\leq (1,2)$.
 
\end{proof}

The lower bound in the case of preorders can be obtained in other ways. For example Yugai \cite{yugai1998} proves that the maximal number of times a diameter can appear in a point set on $n$ points in $\mathbb{R}^{n-2}$ is $\binom{n-1}{2}+n-3<\binom{n}{2}$. Therefore, any preorder that imposes more than this number of diameters will be impossible to attain. For a deeper study on the maximal number of times a diameter of a point set can appear, see \cite{martini2005} and the references therein.

\section{Complete case, upper bounds}
\label{sec:upper}


We offer two proofs for our upper bounds. The first one relies on a powerful lemma by Schoenberg \cite{schoenberg1935} that characterizes families of real numbers that can appear as Euclidean distances induced by a point set. We refer the reader to the text by Matou\v{s}ek \cite{matousek2010} for a nice and short proof of the following result using linear algebra. We also sketch another proof that uses a result by Dekster and Wilker \cite{dekster1987}.

\begin{theorem}
\label{thm:distances}
Let $m_{ij}$, for $i,j$ in $[n+1]$ be nonnegative real numbers with $m_{ij}=m_{ji}$ for all $i,j$ and $m_{ii}=0$ for all $i$. Then there exist points $p_1,\ldots,p_{n+1}$ in $\mathbb{R}^n$ with $\norm{p_i-p_j}=m_{ij}$ for all $i,j$ if and only if the $n\times n$ matrix $G$ with entries $$g_{ij}=\frac{1}{2}(m_{(n+1)i}^2+m_{(n+1)j}^2-m_{ij}^2)$$

for $i,j \in [n]$ is positive semidefinite.
\end{theorem}


Note that Theorem \ref{thm:distances} does not guarantee that the points $p_1,\ldots,p_{n+1}$ are distinct. To illustrate our technique, we begin by proving the upper bound in the case of preorders.

\begin{prop}
\label{prop:upperpre}
Any total preorder $\leq$ on $D_n$ can be induced by a collection of $n$ points in $\mathbb{R}^{n-1}$.
\end{prop}

\begin{proof}
 Let $\equiv$ be the equivalence relation induced on $D_n$ by $\leq$. Since $\leq$ induces a linear order on the equivalence classes, we may name them as follows: $$Q_1<Q_2<\ldots<Q_m.$$
 
 Let $\epsilon>0$ be a sufficiently small real number to be determined later.
 
 We define the following numbers:
 
 \[m_{ij}=\begin{cases}0 & \text{if $i=j$}\\ 1+k \epsilon & \text{if $i<j$ and $(i,j)\in Q_k$}\\
 1+k\epsilon & \text{if $i>j$ and $(j,i) \in Q_k$} \end{cases}\]
 
 From this definition, it is clear that $m_{ij}=m_{ji}$. Consider now the $(n-1)\times (n-1)$ matrix $G$ with entries $$g_{ij}=\frac{1}{2}(m_{ni}^2+m_{nj}^2-m_{ij}^2)$$

for $i,j \in [n]$. We claim that if $\epsilon$ is small enough, then $G$ is positive definite.

Indeed, the values $g_{ij}$ depend continuously on the values $m_{ij}$, and these in turn depend continuously on $\epsilon$. As $\epsilon\to 0$, we get that $$G \to \begin{pmatrix} 2 & 1 & 1 & \cdots & 1 \\
1 & 2 & 1 & \cdots & 1 \\
1 & 1 & 2 & \cdots & 1 \\
 & \vdots & & \ddots & \vdots\\
 1 & 1 & 1 & \cdots & 2\end{pmatrix}$$
 
 The matrix on the right hand side corresponds to the quadratic form $$(x_1,\ldots,x_n)\mapsto 2\left(\sum_{i=1}^n x_i^2 + \sum_{1\leq i<j\leq n} x_ix_j\right)=\left(\sum_{i=1}x_i\right)^2+\sum_{i=1}^n x_i^2,$$
 
 which is positive definite.
 
 The subset of $M_{n}(\mathbb{R})$ consisting of positive definite matrices is open. So we may set $\epsilon>0$ as a number such that $G$ is positive definite.  By Theorem \ref{thm:distances}, there are $p_1,\ldots,p_n$ in $\mathbb{R}^{n-1}$ such that $\norm{p_i-p_j}=m_{ij}$ for all $i,j$. Since distances between distinct points are non zero, it remains to prove that $P = (p_1, p_2, \dots, p_n)$ induces the given preorder $\leq$ on $D_n$.
 
 Indeed, if $(i,j)\leq (k,l)$, then there are indices $m_1\leq m_2$ such that $(i,j)\in Q_{m_1}$ and $(k,l)\in Q_{m_2}$, and then $$\norm{p_i-p_j}=1+m_1\epsilon \leq 1+m_2 \epsilon =\norm{p_k-p_l}.$$

If it is not the case that $(i,j)\leq (k,l)$, then $(k,l)<(i,j)$, so there are indices $m_2<m_1$ such that $(i,j)\in Q_{m_1}$ and $(k,l)\in Q_{m_2}$. Thus, we have

$$\norm{p_i-p_j}=1+m_1\epsilon > 1+m_2 \epsilon =\norm{p_k-p_l},$$

which shows that the we do not have $(i,j)\leq (k,l)$ in the induced relation.

Therefore, we recover exactly the relations given by $\leq$ with the order of pairwise distances of $P$.
 
\end{proof}

An alternative proof can be obtained with the following result of Dekster and Wilker \cite{dekster1987}. For each $n \in \mathbb{N}$, there exists $\lambda(n)$ such that for any collection of $\binom{n+1}{2}$ lengths satisfying $\lambda(n) \leq m_{ij} \leq 1; i, j = 1, \dots, n+1, i \neq j$ there exist points $p_1, \dots, p_{n+1}$ of an $n$-simplex in $\mathbb{R}^n$ such that $m_{ij} = |p_i - p_j|$. Thus, we can set the values of $m_{ij}$ in the interval $[\lambda(n), 1]$ satisfying the prescribed preorder conditions.

A careful adaptation of any of the previous arguments yields the desired result for linear orders. We show how to do this following the ideas in the first proof.

\begin{prop}
\label{prop:upperlinear}
Any linear order $\leq$ on $D_n$ can be induced by a collection of points in $\mathbb{R}^{n-2}$.
\end{prop}

\begin{proof}

Let $\leq$ be a linear order on $D_n$, and let $N=\binom{n}{2}$. Then all the elements in $D_n$ can be listed as follows:

$$(i_1,j_1)<(i_2,j_2)<\ldots<(i_N,j_N).$$

Without loss of generality, we may assume that $(i_1,j_1)=(n-1,n)$.

Let $\epsilon>0$ be a sufficiently small real number to be determined later. For $(i,j)\neq (n-1,n)$ we define the numbers

 \[m_{ij}=\begin{cases}0 & \text{if $i=j$}\\ 1+k \epsilon & \text{if $i<j$ and $(i,j)=(i_k,j_k)$}\\
 1+k\epsilon & \text{if $i>j$ and $(j,i)=(i_k,j_k)$} \end{cases}\]
 
 Now we consider 
 two matrices simultaneously. One is the $(n-2)\times (n-2)$ matrix $G$ with entries 
 
 $$g_{ij}=\frac{1}{2}(m_{(n-1)i}^2+m_{(n-1)j}^2-m_{ij}^2)$$
 
 for $i,j\in[n-2]$. The other one is the $(n-2)\times(n-2)$ matrix $H$ with entries 
 
  $$h_{ij}=\frac{1}{2}(m_{ni}^2+m_{nj}^2-m_{ij}^2)$$
  
  for $i,j\in[n-2]$.
 
 As in the proof of Proposition \ref{prop:upperpre}, there is an $\epsilon_1>0$ such that if $\epsilon<\epsilon_1$, then $G$ is positive definite. Similarly, there is an $\epsilon_2>0$ such that if $\epsilon<\epsilon_2$, then $H$ is positive definite. So if $\epsilon<\min(\epsilon_1,\epsilon_2),$ by Theorem \ref{thm:distances} we can find point sets $p_1,\ldots,p_{n-1}$ and $q_1,\ldots,q_{n-2},q_{n}$ such that \begin{align*}
    \norm{p_i-p_j}&=m_{ij} \quad \text{for $i,j\in [n-1]$}\\
    \norm{q_i-q_j}&=m_{ij} \quad \text{for $i,j\in [n-2]\cup\{n\}$}
 \end{align*}
 
 The point collections $P'=(p_1,\ldots,p_{n-2})$ and $Q'=(q_1,\ldots,q_{n-2})$ have the same pairwise distances, so there is an isometry that takes one to the other. Thus, we may assume that $p_i=q_i$ for $i\in[n-2]$. Let $\pi$ be the hyperplane of $\mathbb{R}^{n-2}$ spanned by $\{p_1,\ldots,p_{n-2}\}$. After possibly a reflection of $q_n$ with respect to $\pi$, we may assume that $p_{n-1}$ and $q_n$ lie on the same halfspace defined by $\pi$.
 
 As $\epsilon \to 0$, $P'$ and $Q'$ converge to be the vertices of a unit regular $(d-2)$-dimensional simplex. Therefore, as $\epsilon \to 0$, we have $\norm{p_{n-1}-q_n}\to 0$. Choose $\epsilon_3$ so that $0 < \norm{p_{n-1}-q_n}<1$ if $0 < \epsilon<\epsilon_3$.
 
 Fix a value of $\epsilon$ such that $\epsilon<\min(\epsilon_1,\epsilon_2,\epsilon_3).$ We set $p_n:=q_n$. Note that $p_n\neq p_{n-1}$ as otherwise we would get the contradiction $$m_{1(n-1)}=\norm{p_1-p_{n-1}}=\norm{p_1-p_n}=m_{1n}.$$
 
 We claim that $P = (p_1, p_2, \dots, p_n)$ induces the linear order $\leq$ on $D_n$.
 
 Indeed, if we have $(i_k,j_k)<(i_l,j_l)$, then $k<l$. If $k\neq 1$ (i.e, we are not considering $(n-1,n)$) then
 
 $$\norm{p_{i_k}-p_{j_k}}=1+k\epsilon<1+l\epsilon = \norm{p_{i_l}-p_{j_l}},$$ so the relation given by $\leq$ is recovered.
 
If $k=1$, we have a relation of the form $(n-1,n)<(i_l,j_l)$, and we note that
 
 $$\norm{p_{n-1}-p_n}<1<1+l\epsilon=\norm{p_{i_l}-p_{j_l}},$$ and this means that the relation is recovered as well.

\end{proof}

\section{Bipartite case, lower bounds} 
\label{sec:BipartiteLower}

For each $n \geq 3$ we show a linear order in $B_{n,n}$ such that no pair of point collections $(P,Q)$ in $\mathbb{R}^{n-2}$ induces the respective linear order. This proves that the minimal dimension $d$ in which every linear order in $B_{r,s}$ is obtainable as pairs of distances from two collections $(P,Q)$ both in $\mathbb{R}^{d}$ satisfies $d > \min(r,s)-2$, since we can consider $n = \min(r,s)$ and a total order in $B_{r,s}$ such that contains the proposed total order in $B_{n,n}$.

We start with the bound for linear orders.

\begin{prop}
    There are no two collections of $n$ points each, in $\mathbb{R}^{n-2}$ that induce a linear order on $B_{n,n}$, including the following relation:
    \begin{itemize}
        \item For each $i \in [n]$, consider 
        \begin{equation}\label{prop:TotalOrderBipartite}
        (i, i) < (i+1, i) < (i+2, i) < \dots < (i-1, i)
        \end{equation}
        where subscripts are taken $mod \; n$.
    \end{itemize}
\end{prop}

\begin{proof}
    We proceed by contradiction. Assume there exist point collections $P = (p_1, \dots, p_n), Q = (q_1, \dots, q_n)$ in $\mathbb{R}^{n-2}$ that induce the desired linear order in $B_{n,n}$. Let $H_i$ denote the halfspace that contains $p_i$ and is defined by the perpendicular bisector of the segment $p_ip_{i+1}$. 
    
    Note that the set of conditions \ref{prop:TotalOrderBipartite} imply that $q_i$ lies in $\cap_{j \neq i-1} H_j$, so the collection of halfspaces $\{H_1, \dots, H_n\}$ satisfies that any $n-1$ of them have non empty intersection. Thus, $\{H_1, \dots, H_n\}$ is a finite family of convex sets in $\mathbb{R}^{n-2}$ satisfying Helly's theorem hypothesis, therefore $\cap_{j \in [n]}H_j$ is non empty. Finally, let $q \in \cap_{j \in [n]}H_j$, by the definition of the halfspaces $H_i$, we have $\norm{q - p_i} < \norm{q - p_{i+1}}$ for each $i \in [n]$, thus $$\norm{q-p_1} < \norm{q-p_2} < \dots < \norm{q-p_n} < \norm{q-p_1},$$ implying $\norm{q-p_1} < \norm{q-p_1}$ which is a contradiction.
\end{proof}

Now we focus on the bound for total preorders. In order to do this, we first show how order conditions on the distances can impose affine independence requirements for the set of points.

\begin{lemma}\label{lm:LowerBoundBipAffine}
    Let $n, d$ be positive integers, and $P = (p_1, \dots, p_n), Q = (q_1, \dots, q_n)$ two collections of points in $\mathbb{R}^{d}$. Suppose the following relations are satisfied:
    \begin{itemize}
        \item For $i = 1,2$, $\norm{p_i- q_1} =  \norm{p_i - q_2} = \dots = \norm{p_i - q_n}$.
        \item For each $i \in \{3, \dots, n-1\}$, 
        $$\norm{p_i-q_1} = \norm{p_i-q_2} = \dots = \norm{p_i - q_{n+2-i}} < \norm{p_i - q_{n+3-i}} = \dots = \norm{p_i - q_n}$$
        \item $\norm{p_n - q_1} < \norm{p_n - q_2} < \dots < \norm{p_n - q_n}$.
        \item $\norm{p_1 - q_1} < \norm{p_2 - q_1}$.
    \end{itemize}
    
    Then the points $p_1, \dots, p_n$ are affinely independent.
\end{lemma}



\begin{proof}
     Note that the second condition implies $q_1, \dots, q_n$ are disctinct, while the third condition ensures $p_1$ different from $p_2$. Thus, $p_1, p_2$ are affinely independent and the line spanned by these two points lies in the subspace given by the intersection of the perpendicular bisectors of $q_1, \dots, q_n$. 
    \begin{figure}[h!]
    \definecolor{xdxdff}{rgb}{0.49019607843137253,0.49019607843137253,1}
    \definecolor{ududff}{rgb}{0.30196078431372547,0.30196078431372547,1}
    \begin{tikzpicture}[line cap=round,line join=round,>=triangle 45,x=1cm,y=1cm]
    \clip(-12.45,-0.45) rectangle (9.49,4.6);
    \draw [line width=0.7pt,dotted,domain=-20.05:9.49] plot(\x,{(-18.291-3.9*\x)/3.9});
    \draw [line width=1pt] (-6.27,2.94)-- (-7.49,0.06);
    \draw [line width=1pt] (-7.49,0.06)-- (-5.11,1.1);
    \draw [line width=1pt] (-5.11,1.1)-- (-6.27,2.94);
    \begin{scriptsize}
    \draw [fill=ududff] (-6.27,2.94) circle (1.5pt);
    \draw[color=ududff] (-6.05,3.10) node {$q_1$};
    \draw [fill=ududff] (-7.49,0.06) circle (1.5pt);
    \draw[color=ududff] (-7.77,0.05) node {$q_3$};
    \draw [fill=ududff] (-5.11,1.1) circle (1.5pt);
    \draw[color=ududff] (-4.80,1.07) node {$q_2$};
    \draw [fill=ududff] (-8.53,3.84) circle (1.5pt);
    \draw[color=ududff] (-8.31,4.1) node {$p_2$};
    \draw [fill=ududff] (-4.63,-0.06) circle (1.5pt);
    \draw[color=ududff] (-4.91,-0.21) node {$p_1$};
    \draw [fill=xdxdff] (-6.02,1.33) circle (0.9pt);
    \draw [fill=ududff] (-4.41,3.72) circle (1.5pt);
    \draw[color=ududff] (-4.19,3.89) node {$p_3$};
    \end{scriptsize}
    \end{tikzpicture}
    \caption{Case n = 3.}
    \end{figure}

    We claim that for each $k \in [n-1]$, the points $p_1, \dots, p_k$ are affinely independent and its spanned affine subspace lies in the subspace given by the intersection of the pairwise perpendicular biscetors of the points $q_1, \dots, q_{n+2-k}$. We proceed by induction on $k$. Our base case is given above, for $k = 2$. 
    
    Assume we have shown the case $k = m$. Now we prove both statements for $k = m+1$. The second one follows immediately, since the first condition gives us $\norm{p_{m+1}-q_{i}} = \norm{p_{m+1}-q_{j}}$ for all $i, j \in [n+1-m]$, so $p_{m+1}$ lies in the intersection of the perpendicular bisectors of points $q_1, \dots, q_{n+1-m}$. Therefore, the affine subspace spanned by $p_1, \dots, p_{m+1}$ lies in the intersection of the desired perpendicular bisectors. 
    
    Now, assuming $p_{m+1}$ lies in the affine subspace spanned by $p_1, \dots, p_{m}$ leads to a contradiction since it will lie in the intersection of all pairwise perpendicular bisectors of the points $q_1, \dots, q_{n+2 - m}$, implying $\norm{p_{m+1} - q_{n+2-m}} = \norm{p_{m+1} - q_{i}}$ for all $i \in [n+1-m]$, contradicting the second condition of our statement. Therefore, $p_{m+1}$ must be affinely independent from $p_1, \dots, p_m$, which concludes our induction.
    
    Thus, $p_1, \dots, p_k$ are affinely independent for each $k \in [n-1]$, and its spanned affine subspace lies in the intersection of the perpendicular bisectors of the points $q_1, \dots, q_{n+2-k}$.
    
    Finally, since $p_1, \dots, p_{n-1}$ are affinely independent and lie in the perpendicular bisector of $q_1, q_2$, the condition $\norm{p_n - q_1} < \norm{p_n - q_2}$ implies $p_1, \dots, p_n$ are affinely independent, as desired.
\end{proof}


We are now ready to provide the construction for our lower bound.

\begin{prop}
    There are no two collections of $n$ points each, in $\mathbb{R}^{n-1}$ such that they induce a total preorder on $B_{n,n}$, including the following relations:
    \begin{itemize}
        \item For $i = 1,2$, $(i, 1) \equiv (i, 2) \equiv \dots \equiv (i, n)$.
        \item For each $i \in \{3, \dots, n-1\}$, $$(i, 1) \equiv (i, 2) \equiv \dots \equiv (i, n+2-i) < (i, n+3-i) \equiv \dots (i,n).$$
        \item $(n, 1) < (n,2) < \dots < (n,n)$.
        \item $(1,1) < (2,1)$.
    \end{itemize}
    Where $(i,j) \equiv (r,s)$ means we have $(i,j) \leq (r,s)$ and $(r,s) \leq (i,j)$. And $(i,j) < (r,s)$ means $(i,j) \leq (r,s)$ without $(r,s) \leq (i,j)$.
\end{prop}

\begin{proof}
    Suppose there exists $P = (p_1, \dots, p_n), Q = (q_1, \dots, q_n)$ that induce the desired total preorder. These collections of points will satisfy all hypothesis stated in \ref{lm:LowerBoundBipAffine}. Therefore $p_1, \dots, p_{n-1}$ are affinely independent, and the affine subspace spanned by these points lie in the affine subspace given by the intersection of the perpendicular bisectors of points $q_1, q_2, q_3$. 
    
    Note that the affine subspace spanned by $p_1, \dots, p_{n-1}$ has dimension $n-2$, thus, by the dimension theorem we conclude that the affine subspace generated by $q_1, q_2, q_3$ has dimension at most $1$ since these subspaces are orthogonal to each other in $\mathbb{R}^{n-1}$. Equivalently,  $q_1, q_2, q_3$ qre collinear. Finally, the intersection of the three perpendicular bisectors defined by the points $q_1, q_2, q_3$ must be non empty, which happens only if at least two of them are equal, contradicting that $q_1, \dots, q_n$ are pairwise distinct. 
\end{proof}

\section{Bipartite case, upper bounds}
\label{sec:BipartiteUpper}

In this section we will adapt our arguments given in Section \ref{sec:upper} for the bipartite case. We start by showing every total preorder in $B_{n,m}$ can be attainable by two collections of points in $\mathbb{R}^{\min(n,m)}$.

\begin{prop}
    Any total preorder $\leq$ on $B_{n,m}$ can be induced by two point collections $P, Q$ of $n, m$ points in $\mathbb{R}^{\min(n,m)}$ respectively.
\end{prop}

\begin{proof}
    
    Assume $\min(n,m) = n$. We proceed in a similar fashion as in Proposition \ref{prop:upperpre}. For each $i \in [m]$ there exists $\epsilon_i$ small enough such that there exist points $p_1^{i}, \dots, p_n^{i}, q_i$ in $\mathbb{R}^{n}$ such that:
    $$\norm{p_j^{i} - p_k^i} = 1+\epsilon_i,$$
    $$\norm{p_j^i - q_i} = 1 + r\epsilon_i$$
    Where $r$ is the position of the class of $(j,i)$ in the prescribed total preorder.
    
    Note that $P^i = \{p_1^{i}, \dots, p_n^i\}$ form a regular simplex for each $i$, with side length $1 + \epsilon_i$. We can assume that the points $p_1^{i}, \dots, p_n^{i}, q_i$ were constructed using the same $\epsilon$ satisfying $0 < \epsilon < \min(\epsilon_1, \dots, \epsilon_m)$. Thus, the point collections $P^i$ have the same pairwise distances for $i = 1, \dots, m$, so there are isometries taking each one to the same simplex. Under these mappings, we obtain $P = (p_1, \dots, p_n)$, $Q= (q_1, \dots, q_m)$ in $\mathbb{R}^{n}$ that induce the desired preorder in $B_{n,m}$.


\end{proof}

Since any linear order is also a total preorder, we immediately have the following.

\begin{coro}
    Any linear order $<$ on $B_{n,m}$ can be induced by two collections $P, Q$ of $n, m$ points in $\mathbb{R}^{\min(n,m)}$ respectively.
\end{coro}

\comentario

\section{Discussion}

We have shown that $d-2$ is the minimal dimension into which every total preorder on $D_n$ can be induced by the order of the pairwise distances of $n$. In the case of linear orders on $D_n$, this minimal dimension can be reduced to $n-3$.

For the bipartite variant, we completely solve the total preorder case, where we show that the minimal dimension into which every total preorder on $B_{m,n}$ can be induced by the order of the pairwise distances of two collections $P$ and $Q$ of $m$ and $n$ points respectively in $\mathbb{R}^{\min(m,n)}$. In the case of linear orders, we reduce the minimal dimension to only two possibilities: either $\min(m,n)$ or $\min(m,n)-1$.

The proof of Proposition \ref{prop:upperlinear} may give the impression that the hypothesis can be weakened to only require a total preorder with a unique minimal element, to which we will associate the distance $\norm{p_{n-1}-p_n}$. The proof will fail, as witnessed by the counterexample in the remark on diameters after the proof of Proposition \ref{prop:lowerpre}.

The main problem when replicating the argument is that there will be no guarantee that the constructed points are distinct.

Also, it is not possible to adapt the argument for Proposition \ref{prop:upperlinear} to the bipartite case directly in order to show that dimension $\min(n,m)-1$ suffices. The problem is that the naïve construction does not offer a way to compare the minimal distances obtained for constructions corresponding to different $q_i$'s. This leaves the following open problem.

\begin{prob}
For $n\geq 2$ an integer determine if the minimal dimension into which any linear order $B_{n,m}$ can be induced by the pairwise distances between two point collections in $\mathbb{R}^d$ is $\min(n,m)-1$ or $\min(n,m)$.
\end{prob}

Finding the minimal dimension into which every total preorder or linear order is achievable is only a first step. The following much more general problems can be studied.

\begin{prob}
For each $n$ and $d$, characterize the linear orders (resp. total preorders) on $D_n$ that can be induced from the ordering of pairwise distances of a point collection of size $n$ in $\mathbb{R}^d$.
\end{prob}

\begin{prob}
For $m,n,d$, characterize the linear orders (resp. total preorders) on $B_{m,n}$ that can be induced from the ordering of pairwise distances of point collections $P$ and $Q$ in $\mathbb{R}^d$ of sizes $m$ and $n$ respectively.
\end{prob}

We expect a full characterization to be out of reach of current tools in the area, as e.g. a full solution to the total preorders problem in the complete case would imply a solution to the maximum number of diameter pairs problem \cite{martini2005}. Nevertheless, any partial progress would shed additional light on the complex behaviour of pairwise distances of points in Euclidean space.



\section{Acknowledgements}

This work was supported by UNAM-PAPIIT IA104621. We thank an anonymous referee for  valuable comments, the reference to the paper of Dekster and Wilker, and the suggestion to study the bipartite case.


\begin{thebibliography}{1}

\bibitem{dekster1987}
B.V. Dekster and J.B. Wilker.
\newblock Edge lengths guaranteed to form a simplex.
\newblock {\em Arch. Math.}, 49:351--366, 1987.

\bibitem{erdos1946sets}
Paul Erd{\"o}s.
\newblock On sets of distances of $n$ points.
\newblock {\em The American Mathematical Monthly}, 53(5):248--250, 1946.


\bibitem{maldonado25}
Gerardo L. Madonado, Edgardo Roldan Pensado and Miguel Raggi.
\newblock Total orders realizable as the distances between two sets of points.
\newblock {\em Discrete Applied Mathematics}, 378:755 -- 761, 2026.
 

\bibitem{martini2005}
H.~Martini and V.~Soltan.
\newblock Antipodality properties of finite sets in {E}uclidean space.
\newblock {\em Discrete Mathematics}, 290(2):221 -- 228, 2005.

\bibitem{matousek2010}
Ji{\v{r}}{\'\i} Matou{\v{s}}ek.
\newblock {\em Thirty-three miniatures: Mathematical and Algorithmic
  applications of Linear Algebra}.
\newblock American Mathematical Society Providence, RI, 2010.

\bibitem{schoenberg1935}
I.J. Schoenberg.
\newblock Remarks to {M}aurice {F}réchet’s article '{S}ur la definition
  axiomatique d’une classe d’espace distances vectoriellement applicable
  sur l’espace de {H}ilbert'.
\newblock {\em Annals of Mathematics}, 36(3):724 -- 732, 1935.

\bibitem{yugai1998}
S.A. Yugai.
\newblock On the largest number of diameters of a point set in {E}uclidean
  space.
\newblock {\em Investigations in Topological and Generalized Spaces (Russian)},
  pages 84--87, 1988.

\end{thebibliography}

\bibliographystyle{plain}

\bigskip

\noindent
{\sc V\'ictor Hugo Almendra-Hern\'andez}
\smallskip

\noindent
{\em Facultad de Ciencias, Universidad Nacional Aut\'onoma de M\'exico, Ciudad de M\'exico, M\'exico}
\smallskip

\noindent
e-mail address: \texttt{vh.almendra.h@ciencias.unam.mx}

\bigskip

\noindent
{\sc Leonardo Mart\'inez-Sandoval}
\smallskip

\noindent
{\em Facultad de Ciencias, Universidad Nacional Aut\'onoma de M\'exico, Ciudad de M\'exico, M\'exico}
\smallskip

\noindent
e-mail address: \texttt{leomtz@ciencias.unam.mx}

\end{document}